\documentclass[review]{elsarticle}
\usepackage{lineno,hyperref}
\usepackage{array}
\usepackage[all]{xy}
\usepackage{color}
\usepackage{amsxtra, microtype}
\usepackage{amssymb,amsmath,amsthm,stmaryrd,latexsym,wasysym}
\usepackage{amstext}
\usepackage{dsfont}
\usepackage{graphicx}
\usepackage{leftidx}
\usepackage[version=3]{mhchem}
\newtheorem{thm}{Theorem}[section]
\newtheorem{cor}[thm]{Corollary}
\newtheorem{lem}[thm]{Lemma}
\newtheorem{prop}[thm]{Proposition}
\newtheorem{defn}[thm]{Definition}
\newtheorem{rem}[thm]{Remark}
\newtheorem{exm}[thm]{Example}
\numberwithin{equation}{section}
\newcommand{\duer}{\mathbin{\raisebox{3pt}{\varhexstar}\kern-3.70pt{\rule{0.15pt}{4pt}}}\,}

\modulolinenumbers[5]

\journal{ }

\begin{document}

\begin{frontmatter}

\title{Representation up to Homotopy of Hom-Lie Algebroids}
\author{S. Merati}

\author{M. R. Farhangdoost\corref{mycorrespondingauthor}}
\cortext[mycorrespondingauthor]{Corresponding author}

\address{Department of Mathematics, College of Sciences,\\Shiraz University, P.O. Box 71457-
44776, Shiraz, Iran\\
farhang@shirazu.ac.ir}

\begin{abstract}
A hom-Lie algebroid is a vector bundle together with a Lie algebroid like structure which is twisted by a homomorphism.
In this paper we use the idea of representations up to homotopy of Lie algebroids to construct a same structure for hom-Lie algebroids and we will explain how representations up to homotopy of length 1 are related to
extensions of hom-Lie algebroids.
\end{abstract}

\end{frontmatter}

\linenumbers
\section{Introduction}
Hom-Lie algebras was introduced by Hartwig, Larsson, and Silvestrov
in \cite{10} as part of a study of deformations of the Witt and Virasoro algebras. In
a hom-Lie algebra, the Jacobi identity is twisted by a linear map, called the hom-
Jacobi identity. In recent years, hom-structures including hom-Lie algebras, $n$-hom-Lie algebras, hom-algebras, hom-coalgebras,
hom-modules and hom-Hopf modules were widely studied, \cite{12,8,7}. The concept of hom-Lie algebroid was introduced by Camille Laurent-Gengoux and Joana Teles in \cite{11}. We study the theory of representation of hom-Lie algebroids as a generalization of Lie algebroids representation.

A significant problem with the usual notion of Lie algebroid representation is the lack of a
well-defined adjoint representation. The effort to resolve this problem has led to a number of
proposed generalizations of the notion of Lie algebroid representation, with the
most popular being that of representation up to homotopy \cite{2,9,3,4,5}.

At first let us recall some notion of Lie algebroids and their representations.
\begin{defn}
\cite{1} A Lie algebroid over a manifold $M$ consists of a vector
bundle $A\rightarrow M$ together with a bundle map $\rho : A \rightarrow TM$ and a Lie bracket $[.,.]$
on the space $\Gamma(A)$, satisfying the Leibniz identity
$$[X, fY] = f[X, Y] + \mathcal{L}_{\rho(X)}(f)Y,$$
for all $X, Y \in\Gamma(A)$ and all $f\in C^{\infty}(M)$.
\end{defn}
Bundle map $\rho$ is called anchor and induces a Lie homomorphism from $\Gamma(A)$ to $\mathcal{X}(M)$.

There is an associated De-Rham-cohomology complex $\Omega(A)=\Gamma(\bigwedge A^*)$ for given Lie algebroid $A\rightarrow M$ with degree $1$ operator
\begin{align*}
d_A\omega(X_1\cdots,X_p)&=\sum_{i<j}(-1)^{i+j}\omega([X_i,X_j],X_1\cdots,\hat{X_i},\cdots,\hat{X_j},\cdots,X_{p})\\
&+\sum_i (-1)^{i+1}\mathcal{L}_{\rho(X_i)}\omega(X_1,\cdots,\hat{X_i},\cdots, X_p).
\end{align*}
The operator $d_A$ is a differential and satisfies the derivation rule, i.e.
$$d_A(\omega\eta)=d_A(\omega)\eta+(-1)^p\omega d_A(\eta),$$
for all $\omega\in\Omega^p(A)$ and $\eta\in\Omega^q(A)$.
\begin{defn}\label{rep}
\cite{4} Let $A$ be a Lie algebroid over $M$, a representation of $A$ is a vector bundle $E\rightarrow M$ together with an $\mathbb{R}$-bilinear map $\nabla:\Gamma(A)\times\Gamma(E)\rightarrow\Gamma(E)$ such that
$$\nabla_{fX}s=f\nabla_Xs,\quad\nabla_X fs=f\nabla_{X}s+\mathcal{L}_{\rho(X)}(f)(s)$$
and
$$\nabla_{[X,Y]}s=\nabla_X\nabla_Ys-\nabla_Y\nabla_Xs,$$
for all $f\in C^{\infty}(M)$, $X,Y\in \Gamma(A)$ and $s\in\Gamma(E).$
\end{defn}
The $\mathbb{R}$-bilinear map $\nabla$ in Definition \ref{rep} is called a flat $A$-connection on $E$.
\begin{prop}
\cite{2} Given a Lie algebroid $A$ and a vector bundle $E$ over $M$, there is a 1-1
correspondence between flat $A$-connections $\nabla$ on $E$ and degree 1 operators $d_{\nabla}$ on
$\Omega(A;E)=\Gamma(\bigwedge A^*\otimes E)$ which
satisfy the derivation rule and $d^2_\nabla = 0$.
\end{prop}
The structure of this paper is as follows:
\begin{itemize}
\item In section two, we study some properties, new examples of hom-algebroids and a relation between morphisms of two hom-Lie algebroids and a hom-Lie sub-algebroids of their direct sum.
\item The third section is devoted to the concept of representation of hom-Lie algebroids.
\item
In last section we introduce the hom-Lie algebroid representations up to homotopy and surveyed some examples. The main result of this paper is Theorem \ref{up}, which is based on Proposition \ref{upto}.
\end{itemize}
\section{Hom-Lie Algebroid}
At first let us to recall the definition of hom-Lie algebroids.
\begin{defn}\cite{11}
A hom-Lie algebroid is a quintuple $( A \to M, \theta,  [.,.]_A,\rho,\Theta) $, where
$A \to M$ is a vector bundle over a manifold $M$, $\theta:M\to M$ is a smooth map,
$[.,.]_A: \Gamma(A) \otimes \Gamma(A) \to \Gamma(A)$ is a bilinear map, called bracket,
$\rho : A \to  TM $ is a vector bundle morphism, called anchor, and
$\Theta:\Gamma(A) \to \Gamma(A) $ is a linear endomorphism  of $\Gamma(A) $  such that
\begin{enumerate}
  \item  $\Theta(fX) = \theta^* (f) \Theta (X)$,
for all $X \in \Gamma(A), f \in C^\infty(M)$;
  \item
the triple $(\Gamma(A),[.,.]_A, \Theta)$ is a hom-Lie algebra;
   \item the following hom-Leibniz identity holds:
 $$ [X,fY]_A= \theta^*(f) [X,Y]_A + \mathcal{L}_{\rho(X)} (f) \Theta(Y), \quad \mbox{for all } X, Y \in \Gamma(A), f \in C^\infty(M). $$
    \item[4.] $(\Theta,\theta^*)$ is a representation of $(\Gamma(A),[.,.]_A, \Theta)$ on $C^\infty(M)$.
\end{enumerate}

\end{defn}
The hom-Lie algebroid $A$ is transitive, if $\rho$ is fiberwise surjective and $\theta$ is a submersion. The hom-Lie algebroid $A$ is regular, if $\rho$ and $\Theta$ are of locally constant rank, and totally intransitive, if $\rho=0$.

A morphism $\varphi$ between two hom-Lie algebroid $A$ and $A'$ on same base $M$ is a vector bundle morphism $\varphi:A\rightarrow A'$ such that $\rho'\circ\varphi=\rho$, $\Theta'\circ\varphi^*=\varphi^*\circ\Theta$ and
$\varphi([X,Y]_A)=[\varphi(x),\varphi(Y)]_A$, for all $X,Y\in\Gamma(A)$.
\begin{prop}
Let $A$ be a hom-Lie algebroid on $M$, and let $U\subset M$ be an open subset, such that $\theta(U)\subset U$. Then the bracket $[.,.]_A:\Gamma(A)\times\Gamma(A)\rightarrow\Gamma(A)$ restricts to $[.,.]_U:\Gamma_U(A)\times\Gamma_U(A)\rightarrow\Gamma_U(A)$ and makes $A_U$ a hom-Lie algebroid on $U$, called the restriction of $A$ to $U$.
\end{prop}
\begin{proof}
It is enough to show that $[X,Y]_A$ vanishes on $U$, where $X,Y\in\Gamma(A)$ and $Y$ vanishes on $U$. Let $x\in U$ and $f\in C^{\infty}(M)$ such that $f(x)=0$ and $f(p)=1$ for all $p\in M\backslash U$. Then
\begin{align*}
[X,Y](x)&=[X,fY](x)\\
&=f(\theta(x))[X,Y](x)+\mathcal{L}_{\rho(X)}f(x)\Theta(Y)(x)\\
&=0.
\end{align*}
\end{proof}
\begin{exm}
Let $M$ be a manifold and $(\mathfrak{g},[.,.],\alpha)$ be a hom-Lie algebra on $TM\oplus(M\times\mathfrak{g})$ define an anchor $\rho=\pi_1:TM\oplus(M\times\mathfrak{g})\rightarrow TM$, a bracket
$$[X\oplus v,Y\oplus w]=[X,Y]_{TM}\oplus(X(w)-Y(v)+[v,w]_{\mathfrak{g}}),$$
and vector bundle map
\begin{align*}
\Theta :\Gamma(TM\oplus(M\times\mathfrak{g}))&\rightarrow \Gamma( TM\oplus(M\times\mathfrak{g}))\\
f\oplus g&\mapsto f\oplus \alpha\circ g
\end{align*}
over $M$. Then $TM\oplus(M\times\mathfrak{g})$ is a hom-Lie algebroid, called the trivial hom-Lie algebroid on $M$ with structure hom-Lie algebra $\mathfrak{g}$. $TM\oplus(M\times\mathfrak{g})$ is a transitive hom-Lie algebroid.
\end{exm}
Let $A$ be a vector bundle over $M$. We denote by $\mathcal{D}(A)$ the collection of all isomorphism on $A$ over $id_M$.
\begin{defn}
A hom-Lie algebra bundle is a vector bundle $\xymatrix{L\ar[r]^{\pi}&M}$ together with $\Theta\in\mathcal{D}(L)$ and a field of hom-Lie algebras bracket $[.,.]:\Gamma L\times\Gamma L\rightarrow\Gamma L$ with respect to $\Theta$, that is $[.,.]_m:L_m\times L_m\rightarrow L_x$ is a hom-Lie algebra bracket with respect to $\Theta_m$, for all $m\in M$, Such that $L$ admit an atlas $\{\varphi_i:U_i\times \mathfrak{g}\rightarrow L_{U_i}\}$ in which each $\varphi_{i,m}$ is a hom-Lie algebra isomorphism.
\end{defn}
\begin{prop}
Let $A$ and $A'$ be hom-Lie algebroids on $M$ and let $A$ be transitive. Let $A\oplus_{TM}A'$ denote the inverse image vector bundle over $M$ of the diagram
\begin{equation*}
\xymatrix{A\oplus_{TM}A'\ar[r]\ar[d]&A'\ar[d]^{\rho'}\\
A\ar[r]_{\rho}&TM.
}
\end{equation*}
Let $\bar{\rho}:A\oplus_{TM}A'\rightarrow TM$ be the diagonal composition and define a bracket on $\Gamma(A\oplus_{TM}A')$ by
$$[X\oplus X',Y\oplus Y']_{A\oplus_{TM}A'}=[X,Y]_A\oplus[X',Y']_A',$$
and vector bundle map $\Theta\oplus\Theta':\Gamma(A\oplus_{TM} A')\rightarrow \Gamma(A\oplus_{TM} A').$
Then $A\oplus_{TM}A'$ is a hom-Lie algebroid on $M$ and the diagram above is now a pullback in the category of hom-Lie algebroids over $M$.
We call $A\oplus_{TM}A'$ the direct sum hom-Lie algebroid of $A$ and $A'$ over $TM$.
\end{prop}
\begin{rem}
Note that the trivial hom-Lie algebroid $TM\oplus(M\times\mathbf{g})$ is not a direct sum hom-Lie algebroid of $TM$ and $M\times\mathfrak{g}$.
\end{rem}
A sub-vector bundle $B\rightarrow M$ of $A\rightarrow M$ is a hom-Lie sub-algebroid of $( A, \theta,  [.,.]_A,\rho,\Theta) $ if $\Theta(\Gamma(B))\subseteq \Gamma(B)$ and
$$[X,Y]_A\in\Gamma(B),$$ for all $X,Y\in \Gamma(B)$.
\begin{prop}
If $( A, \theta,  [.,.]_A,\rho,\Theta) $ and $( A', \theta,  [.,.]_{A'},\rho',\Theta') $ are two hom-Lie algebroid over $M$ and $\varphi:A\rightarrow A'$ is a vector bundle morphism, then $\varphi$ is a hom-Lie algebroid morphism if and only if
$$\mathcal{G}_{\varphi}=\{(x,\varphi(x))|x\in A\}\subseteq A\oplus_{TM}A'$$ is a hom-Lie sub-algebroid of $(A\oplus_{TM}A',[.,.]_{A\oplus_{TM}A'},\bar{\rho},\Theta\oplus\Theta')$.
\end{prop}
\begin{proof}
Let $\varphi:(A,[.,.]_A,\rho,\Theta)\rightarrow(A',[.,.]_{A'},\rho',\Theta')$ be a hom-Lie algebroid morphism. Any section of sub-bundle $\mathcal{G}_{\varphi}$ is same as a $(X,\varphi(X))$, where $X$ is a section of $A$. For any $X,Y\in\Gamma(A)$,
\begin{align*}
[(X,\varphi(X)),(Y,\varphi(Y))]_{A\oplus_{TM}A}&=([X,Y]_A,[\varphi(X),\varphi(Y)]_{A'})\\
&=([X,Y]_A,\varphi([X,Y]_A))\in\Gamma(\mathcal{G}_{\varphi}).
\end{align*}
Also
$$\Theta\oplus\Theta'(X,\varphi(X))=(\Theta(X),\Theta'(\varphi(X)))=(\Theta(X),\varphi(\Theta(X)))\in\Gamma(\mathcal{G}_{\varphi}).$$
Thus $\mathcal{G}_{\varphi}$ is a hom-Lie sub-algebroid of $A\oplus_{TM}A'$.

Conversely, if $\mathcal{G}_{\varphi}\subseteq A\oplus_{TM}A'$ is a hom-Lie sub-algebroid of
$$(A\oplus_{TM}A,[.,.]_{A\oplus_{TM}A},\bar{\rho},\Theta\oplus\Theta'),$$ then we have
$$[(X,\varphi(X)),(Y,\varphi(Y))]_{A\oplus_{TM}A}=([X,Y]_A,[\varphi(X),\varphi(Y)]_{A'})\in\Gamma(\mathcal{G}_{\varphi}),$$
which implies that $\varphi([X,Y]_A)=[\varphi(X),\varphi(Y)]_{A'}$, for all $X,Y\in\Gamma(A)$.

Furthermore $\Theta\oplus\Theta'(\Gamma(\mathcal{G}_{\varphi}))\subseteq\Gamma(\mathcal{G}_{\varphi})$, yields that
$$\Theta\oplus\Theta'(X,\varphi(X))=(\Theta(X),\Theta'(\varphi(X))),$$
which implies that $\Theta'\circ\varphi^*(X)=\varphi^*\circ\Theta(X)$. Therefore $\varphi$ is a morphism of hom-Lie algebroids.
\end{proof}

\section{Representation up to homotopy of hom-Lie algebroids}
Let $A$ be a hom-Lie algebroid over $M$. An $A$-connection on a vector bundle $E$ over $M$ with respect to an isomorphism $\alpha\in \mathcal{D}(E)$ is an $\mathbb{R}$-bilinear map $\nabla^{\alpha}:\Gamma(A)\times\Gamma(E)\rightarrow\Gamma(E)$, $(X,s)\mapsto\nabla^{\alpha}_X s$ such that
$$\nabla^{\alpha}_fX s=f\nabla^{\alpha}_X s,\qquad\nabla^{\alpha}_X fs=\theta^*f\nabla^{\alpha}_X s+\mathcal{L}_{\rho(X)}(f) s,$$
and $$\nabla^{\alpha}_{\Theta(X)}\alpha(s)=\alpha(\nabla^{\alpha}_X s)$$
for all $f\in C^{\infty}(M)$, $s\in\Gamma(E)$ and $X\in\Gamma(A)$.

The $A$-curvature of $\nabla^{\alpha}$ is the tensor given by
\begin{align*}
R_{\nabla^{\alpha}}(X,Y)s:=\nabla_{\Theta(X)}\nabla_{Y}s-\nabla_{\Theta(Y)}\nabla_X s-\nabla_{[X,Y]}\alpha(s)
\end{align*}
for $X,Y\in\Gamma(A)$ and $s\in\Gamma(E)$. The $A$-connection $\nabla^{\alpha}$ with respect to $\alpha$ is called flat, if $R_{\nabla^{\alpha}}=0$.
\begin{prop}
Let $\xymatrix{A\ar[r]&M}$ be a vector bundle which $A_m$ be a hom-Lie algebra by hom-function $\Theta_m$, where $\Theta\in\mathcal{D}(A)$ and $m\in M$. Then $\xymatrix{A\ar[r]&M}$ is a hom-Lie algebra bundle if and only if it is admits a flat $A$-connection $\nabla^{\Theta}$ over $A$ with respect to $\Theta$.
\end{prop}
\begin{proof}
Let there exists a $\nabla^{\Theta}$ with $R_{\nabla^{\Theta}}=0$. We need to prove that the hom-Lie algebra structure on the fibers is locally trivial. The vanishing of the curvature means that $\nabla^{\Theta}$ acts as $\Theta$-derivations of the hom-Lie algebra fibers
$$\nabla_{\Theta(X)}\nabla_{Y}s+\nabla_{\Theta(Y)}\nabla_X s=\nabla_{[X,Y]}\alpha(s).$$
Since $\Theta$ is over $id_M$, $\Theta$-derivations are infinitesimal automorphism, we deduce that parallel transports induced by $\nabla^{\Theta}$ are hom-Lie algebra isomorphism, providing the necessary hom-Lie algebra bundle trivialization.
For the converse, assume that $A$ is a hom-Lie algebra bundle, so $A$ is locally trivial then locally one can choose connections with zero $A$-curvature. Then one can use partitions of unity to construct a global connection with the same property.
\end{proof}
\begin{defn}\label{reph}
Let $A$ be a hom-Lie algebroid over $M$ and $E$ be a vector bundle over $M$. A representation of $A$ with respect to $\alpha\in\mathcal{D}(E)$ on $E$, is a flat $A$-connection $\nabla^{\alpha}$ on $E$.
\end{defn}
In the case where $A=\mathfrak{g}$ is a hom-Lie algebra, Definition \ref{reph} recovers the standard notion of representation of hom-Lie algebras.
Given an $A$-connection $\nabla^{\alpha}$ with respect to $\alpha\in\mathcal{D}(E)$ on $E$, the space of $\Theta$-compatible $E$-valued differential forms, $$\Omega_{\alpha}(A;E):=\{\omega\in\Gamma(\bigwedge A^*\otimes E)|\alpha\circ\omega=\Theta^*\omega\}$$ has an induced operator $d_{\nabla^{\alpha}}$ given by following formula:
\begin{align}\label{deri}
d_{\nabla^{\alpha}}\omega(X_1,\cdots,X_{p+1})&=\sum_{i<j}(-1)^{i+j}\omega([X_i,X_j]_A,\Theta(X_1)\cdots,\hat{X_i},\cdots,\hat{X_j},\cdots,\Theta(X_{p+1}))\nonumber\\
&+\sum_i(-1)^{i+1}\nabla^{\alpha}_{\Theta^{p}(X_i)}\omega(X_1,\cdots,\hat{X_i},\cdots,X_{p+1}).
\end{align}
\begin{lem}\label{lemder}
 In general, $d_{\nabla^{\alpha}}$ satisfies the $\alpha$-derivation rule:
\begin{equation}
\Theta^*d_{\nabla^{\alpha}}\omega=\alpha\circ d_{\nabla^{\alpha}}\omega,
\end{equation}
and
\begin{equation}\label{der}
d_{\nabla^{\alpha}}(\omega\eta)=d_{\nabla^{\alpha}}(\omega)\Theta^*(\eta)+(-1)^p\Theta^*(\omega)d_{\nabla^{\alpha}}(\eta)
\end{equation}
for any $\omega\in\Omega^p_{\alpha}(A;E)$ and $\eta\in\Omega^q_{\alpha}(A;E)$.
\end{lem}
\begin{proof}
Let $\omega\in\Omega^p_{\alpha}(A;E)$ and $X_i\in\Gamma(A)$ for $1\leq i\leq p+1$,
\begin{eqnarray*}
&&\Theta^*(d_{\nabla^{\alpha}}\omega(X_1,\cdots, X_{p+1}))\\&=&d_{\nabla^{\alpha}}\omega(\Theta(X_1),\cdots,\Theta(X_{p+1}))\\
&=&\sum_{i<j}(-1)^{i+j}\omega([X_i,X_j]_A,\Theta^2(X_1)\cdots,\widehat{\Theta(X_i)},\cdots,\widehat{\Theta(X_j)},\cdots,\Theta^2(X_{p+1}))\\
&+&\sum_i(-1)^{i+1}\nabla^{\alpha}_{\Theta^{p+1}(X_i)}\omega(\Theta(X_1),\cdots,\widehat{\Theta(X_i)},\cdots,\Theta(X_{p+1}))
\end{eqnarray*}
\begin{eqnarray*}
&=&\sum_{i<j}(-1)^{i+j}\Theta^*\omega([X_i,X_j]_A,\Theta(X_1)\cdots,\hat{X_i},\cdots,\hat{X_j},\cdots,\Theta(X_{p+1}))\\
&+&\sum_i(-1)^{i+1}\alpha\circ\nabla^{\alpha}_{\Theta^{p}(X_i)}\Theta^*\omega(X_1,\cdots,\hat{X_i},\cdots,X_{p+1})\\
&=&\sum_{i<j}(-1)^{i+j}\alpha\circ\omega([X_i,X_j]_A,\Theta(X_1)\cdots,\hat{X_i},\cdots,\hat{X_j},\cdots,\Theta(X_{p+1}))\\
&+&\sum_i(-1)^{i+1}\alpha\circ\nabla^{\alpha}_{\Theta^{p}(X_i)}\omega(X_1,\cdots,\hat{X_i},\cdots,X_{p+1})\\
&=&\alpha\circ d_{\nabla^{\alpha}}\omega(X_1,\cdots, X_{p+1}).
\end{eqnarray*}
The proof for (\ref{der}) is similar.
\end{proof}

\begin{thm}\label{main1}
 Given a hom-Lie algebroid $A$ and a vector bundle $E$ over $M$, then there is a $1-1$
correspondence between $A$-connections $\nabla^{\alpha}$ with respect to $\alpha\in\mathcal{D}(E)$ on $E$ and degree $+1$ operators $d_{\nabla^{\alpha}}$ on
$\Omega_{\alpha}(A;E)$ which
satisfy the $\alpha$-derivation rule. Moreover, $\nabla^{\alpha}$ is a representation with respect to $\alpha$ on $E$ if and only if $d_{\nabla^{\alpha}}^2=0$.
\end{thm}
\begin{proof}
The $1-1$ correspondence arise from (\ref{deri}) and Lemma \ref{lemder}.
Now we show that $\nabla^{\alpha}$ is a representation with respect to $\alpha$ on $E$ if and only if $d_{\nabla^{\alpha}}^2=0$.
Let $\omega\in\Omega^p_{\alpha}(A;E)$ and $X_i\in\Gamma(A)$ for $1\leq i\leq p$. By straightforward computations, we have
\begin{eqnarray*}
&&d_{\nabla^{\alpha}}^2\omega(X_1,\cdots, X_{p+2})\\
&=&\sum_{i<j}(-1)^{i+j}d_{\nabla^{\alpha}}\omega([X_i,X_j]_A,\Theta(X_1)\cdots,\hat{X_i},\cdots,\hat{X_j},\cdots,\Theta(X_{p+2}))\\
&+&\sum_i(-1)^{i+1}d_{\nabla^{\alpha}}\nabla^{\alpha}_{\Theta^{p+1}(X_i)}\omega(X_1,\cdots,\hat{X_i},\cdots,X_{p+2}).
\end{eqnarray*}
Also
\begin{eqnarray}
&&d_{\nabla^{\alpha}}\nabla^{\alpha}_{\Theta^{p+1}(X_i)}\omega(X_1,\cdots,\hat{X_i},\cdots,X_{p+2})\nonumber\\
&=&\sum_{l=1}^{i-1}(-1)^{l+1}\nabla^{\alpha}_{\Theta^{p+1}(X_i)}\nabla^{\alpha}_{\Theta^{p}(X_l)}\omega(X_1,\cdots,\hat{X_l},\cdots,\hat{X_i},\cdots,X_{p+2})\label{18}\\
&+&\sum_{l=i+1}^{p+2}(-1)^{l}\nabla^{\alpha}_{\Theta^{p+1}(X_i)}\nabla^{\alpha}_{\Theta^{p}(X_l)}\omega(X_1,\cdots,\hat{X_i},\cdots,\hat{X_l},\cdots,X_{p+2})\label{19}\\
&+&\sum_{i<l<k}(-1)^{l+k}\nabla^{\alpha}_{\Theta^{p+1}(X_i)}\omega([X_l,X_k]_A,\Theta(X_1),\cdots,\widehat{X_{i,l,k}},\cdots,\Theta(X_{p+2}))\nonumber\\
&+&\sum_{l<k<i}(-1)^{l+k}\nabla^{\alpha}_{\Theta^{p+1}(X_i)}\omega([X_l,X_k]_A,\Theta(X_1),\cdots,\widehat{X_{l,k,i}},\cdots,\Theta(X_{p+2}))\nonumber\\
&+&\sum_{l<i<k}(-1)^{l+k+1}\nabla^{\alpha}_{\Theta^{p+1}(X_i)}\omega([X_l,X_k]_A,\Theta(X_1),\cdots,\widehat{X_{l,i,k}},\cdots,\Theta(X_{p+2})),\nonumber
\end{eqnarray}
and\\
\begin{eqnarray}
&&d_{\nabla^{\alpha}}\omega([X_i,X_j]_A,\Theta(X_1)\cdots,\hat{X_i},\cdots,\hat{X_j},\cdots,\Theta(X_{p+2}))\nonumber\\
&=&\nabla^{\alpha}_{\Theta^p[X_i,X_j]_A}\omega(\Theta(X_1),\cdots,\hat{X_i},\cdots,\hat{X_j},\Theta(X_{p+2}))\label{20}\\
&+&\sum_{l=1}^{i-1}(-1)^l\nabla_{\Theta^{p+1}(X_l)}\omega([X_i,X_j]_A,\Theta(X_1)\cdots,\widehat{X_{l,i,j}},\cdots,\Theta(X_{p+2}))\nonumber\\
&+&\sum_{l=i+1}^{j-1}(-1)^{l+1}\nabla_{\Theta^{p+1}(X_l)}\omega([X_i,X_j]_A,\Theta(X_1)\cdots,\widehat{X_{i,l,j}},\cdots,\Theta(X_{p+2}))\nonumber\\
&+&\sum_{l=j+1}^{p+2}(-1)^{l}\nabla_{\Theta^{p+1}(X_l)}\omega([X_i,X_j]_A,\Theta(X_1)\cdots,\widehat{X_{i,l,j}},\cdots,\Theta(X_{p+2}))\nonumber\\
&+& \sum_{l=1}^{i-1}(-1)^{l+1}\omega([[X_i,X_j]_A,\Theta(X_l)]_A,\Theta^2(X_1),\cdots,\widehat{X_{l,i,j}},\cdots,\Theta^2(X_{p+2}))\label{21}\\
&+&\sum_{l=i+1}^{j-1}(-1)^{l}\omega([[X_i,X_j]_A,\Theta(X_l)]_A,\Theta^2(X_1),\cdots,\widehat{X_{l,i,j}},\cdots,\Theta^2(X_{p+2}))\label{22}\\
&+&\sum_{l=j+1}^{p+2}(-1)^{l+1}\omega([[X_i,X_j]_A,\Theta(X_l)]_A,\Theta^2(X_1),\cdots,\widehat{X_{i,j,l}},\cdots,\Theta^2(X_{p+2}))\label{23}\\
&+&\sum_{l,k}(\pm)\omega([\Theta(X_l),\Theta(X_k)]_A,\Theta([X_i,X_j],\Theta^2(X_1),\cdots,\widehat{X_{i,j,l,k}},\cdots,\Theta^2(X_{p+2})).\qquad\quad\label{24}
\end{eqnarray}
By $\widehat{X_{i_1,\cdots,i_n}}$ we mean that we omit the items $X_{i_1},\cdots,X_{i_n}$, where $n\in\mathbb{N}$.
Whereas $[\Theta(X_l),\Theta(X_k)]_A=\Theta([X_l,X_k])$ and
by the hom-Jacobi identity, we obtain that
$$\sum_{i<j}(-1)^{i+j}(\mbox{Eq.~}\ref{24}+\mbox{Eq.~}\ref{23}+\mbox{Eq.~}\ref{22}+\mbox{Eq.~}\ref{21})=0,$$
also
$$\sum_{i=1}^{p+2}(-1)^{i+1}(\mbox{Eq.~}\ref{18}+\mbox{Eq.~}\ref{19})+\sum_{i<j}(-1)^{i+j}(\mbox{Eq.~}\ref{20})=0$$
since $\nabla^{\alpha}$ is a representation with respect to $\alpha$ on $E$ and $\Theta^*\omega=\alpha\circ\omega$.
Then $d_{\nabla^{\alpha}}^2=0$.
\end{proof}

\begin{exm}
Every regular hom-Lie algebroid $A$, $A$ has two canonical representations. They are
$\mathfrak{g}(A):=ker(\rho)$, together with the $A$-connection with respect to $\Theta$
$$\nabla^{adj,\Theta}_X (Y) = [X, Y]_A,$$
and the normal bundle $\mathcal{V}(A) := TM/\rho(A)$ of the foliation induced by $A$, together with the $A$-connection with respect to $\theta_*$
$$\nabla^{adj,\theta_*}_X (\overline{Y}) =\overline{[\rho(X),Y]},$$
where $\overline{X} = X $ mod $\rho(A)$.
\end{exm}
\section{Representation up to homotopy of hom-Lie algebroids}
Now, using the contents expressed, we introduce the main concept of this paper.
\begin{defn}
A representation up to homotopy of hom-Lie algebroid $A$ on a graded vector bundle $\varepsilon$ with respect to degree preserving operator $\alpha$ on $\varepsilon$, is a degree $1$ operator $D_{\alpha}$ on $\Omega_{\alpha}(A;\varepsilon)$ such that $D_{\alpha}^2=0$,
\begin{equation}\label{41}
\Theta^*D_{\alpha}=\alpha\circ D_{\alpha}
\end{equation} and
\begin{equation}\label{42}
D_{\alpha}(\omega\eta)=D_{\alpha}\omega\Theta^*(\eta)+(-1)^p\Theta^*\omega D_{\alpha}(\eta),
\end{equation}
for any $\omega\in\Omega^p(A)$ and $\eta\in\Omega_{\alpha}(A;\varepsilon)$.
\end{defn}
By an $\alpha$-representation up to homotopy we mean a representation up to homotopy with respect to $\alpha$.
\begin{prop}\label{upto}
There is a $1-1$ correspondence between $\alpha$-representations up to homotopy
$(\varepsilon,D_{\alpha})$ of $A$ and graded vector bundles $\varepsilon$ over $M$ endowed with:
\begin{enumerate}
\item A degree $1$ operator $\partial_{\alpha}$ on $\varepsilon$ making $(\varepsilon, \partial_{\alpha})$ a complex.
\item An $A$-connection $\nabla^{\alpha}$ on $(\varepsilon, \partial_{\alpha})$ with respect to $\alpha$.
\item An $End(\varepsilon)$-valued $2$-form $\omega_2$ of total degree $1$, i.e.
$$\omega_2\in\Omega^2(A;\underline{End}^{-1}(\varepsilon))$$
satisfying
\begin{equation}\label{43}
\omega_2(\Theta(X_1),\Theta(X_2))=\omega_2(X_1,X_2)\circ\alpha
\end{equation}
 and
 \begin{equation}\label{r}
\partial_{\alpha}(\omega_2)+R_{\nabla^{\alpha}}=0.
 \end{equation}

\item For each $p > 2$, an $End(\varepsilon)$-valued $p$-form $\omega_p$ of total degree $1$, i.e.
$$\omega_p\in\Omega^p(A;\underline{End}^{1-p}(\varepsilon))$$
satisfying
\begin{equation}\label{45}
\omega_p(\Theta(X_1),\Theta(X_2),\cdots,\Theta(X_p))=\omega_p(X_1,X_2,\cdots,X_p)\circ\alpha
\end{equation}
and
$$\partial_{\alpha}(\omega_p)+d_{\nabla^{\alpha}}(\omega_{p-1})+\omega_2\circ\omega_{p-2}+\omega_3\circ\omega_{p-3}+\cdots\omega_{p-2}\circ\omega_2=0.$$
\end{enumerate}
Moreover, the correspondence is characterized by
\begin{equation}\label{cha}
D_{\alpha}(\eta)=\partial_{\alpha}(\eta)+d_{\nabla^{\alpha}}(\eta)+\omega_2\wedge\eta+\omega_3\wedge\eta+\cdots.
\end{equation}
\end{prop}
\begin{proof}
Owing to (\ref{42}) and the fact that $\Omega_{\alpha}(A;\varepsilon)$ is generated as an $\Omega(A)$-module by $\Gamma(\varepsilon)$, the operator $D_{\alpha}$ will be uniquely determined by what it does on $\Gamma(\varepsilon)$ it will send each $\Gamma(\varepsilon^k)$ into the sum
$$\Gamma(\varepsilon^{k+1})\oplus\Gamma^1_{\alpha}(A;\varepsilon^k)\oplus\Gamma^2(A;\varepsilon^{k-1})\oplus\cdots.$$
Then it will also send each $\Omega^p_\alpha(A;\varepsilon^k)$ into the sum
$$\Omega^p_{\alpha}(A;\varepsilon^{k+1})\oplus\Omega^{p+1}_{\alpha}(A;\varepsilon^k)\oplus\Omega_{\alpha}^{p+2}(A;\varepsilon^{k-1})\oplus\cdots,$$
and we denote by $D_0,D_1,\cdots$ the components of $D$. By (\ref{42}), we deduce that each $D_i$ for $i\neq 1$ is a (graded) $\Omega(A)$-linear map and, by \cite[Lemma A.1]{2} it is the wedge product with an element in $\Omega_{\alpha}(A;End(\varepsilon))$. On the other hand $D_1$ satisfies the $\alpha$-derivation rule on each the vector bundle $\varepsilon^k$ and by Theorem \ref{main1}.
it comes from $A$-connections on these bundles. the equations (\ref{43}), (\ref{r}) and (\ref{45}) correspond to (\ref{41}) and $D^2=0$.
\end{proof}
A morphism $\varphi:\varepsilon_1\rightarrow\varepsilon_2$ between $\alpha$-representation up to homotopy $(\varepsilon_1,D_{\alpha})$ and $\beta$-representation up to homotopy $(\varepsilon_2,D_{\beta})$ of hom-Lie algebroid $A$ is a degree zero $\Omega(A)$-linear map
$$\varphi:\Omega_{\alpha}(A;\varepsilon_1)\rightarrow\Omega_{\beta}(A;\varepsilon_2)$$
which commutes with $\alpha$ and $\beta$ and the structure differentials $D_{\alpha}$ and $D_{\beta}$.
\begin{cor}
Any $\alpha$-representation up to homotopy $(\varepsilon,D_{\alpha})$ is isomorphic whit $\alpha$-representation up to homotopy $(\varepsilon,\overline{D_{\alpha}})$, which
$$\overline{D_{\alpha}}(\eta)=-\partial_{\alpha}(\eta)+d_{\nabla^{\alpha}}(\eta)-\omega_2\wedge\eta+\omega_3\wedge\eta+\cdots,$$
where $D_{\alpha}$ is characterized as (\ref{cha}).
\end{cor}
\begin{proof}
Let $\varphi:\Omega_{\alpha}(A;\varepsilon)\rightarrow\Omega_{\alpha}(A;\varepsilon)$ equal to $(-1)^k$ on $\varepsilon^k$
and $\eta\in \Omega_{\alpha}(A;\varepsilon)$,
it is easy to show that
$$\overline{D_{\alpha}}(\varphi(\eta))=\varphi(D_{\alpha}(\eta)).$$
Moreover, $\varphi\circ\alpha=\alpha\circ\varphi$, since $\alpha$ is a degree preserving operator on $\varepsilon$.
\end{proof}
\begin{exm}
Any representation $E$ of hom-Lie algebroid $A$ with respect to $\alpha\in\mathcal{D}(E)$ can be seen
as a representation up to homotopy with respect to $\alpha$ concentrated in degree zero.
\end{exm}
\begin{exm}
Let $\alpha\in\mathcal{D}(M\times\mathbb{R})$ and $\omega\in\Omega_{\alpha}^n(A)$ be a closed $n$-form such that
$\Theta^*\omega=\alpha\circ\omega.$
Then $\omega$ induces a representation up to homotopy on the complex which is the trivial line bundle in degrees $0$ and $n-1$, and
zero otherwise. The structure operator is $\nabla+\omega$, where $\nabla$ is the flat connection on
the trivial line bundle. If $\omega$ and $\omega'$ are cohomologous, then the resulting representations up
to homotopy are isomorphic with isomorphism defined by $Id+\theta$
where $\omega-\omega'=d\theta$.
\end{exm}
\begin{thm}\label{up}
For any $\alpha$-representation up to homotopy of length one with vector bundles
$E$ in degree $0$ and $F$ in degree $1$ and structure operator $D_{\alpha} = \partial+\nabla^{\alpha}+K$, there is an extension of hom-Lie algebroids
$$Hom(E,F)\rightarrow\tilde{A}\rightarrow A,$$
where
\begin{enumerate}
\item $Hom(E,F)$ is a bundle of Lie algebras with bracket $$[f,g]_{\partial}=f\partial g-g\partial f.$$
\item $\tilde{A}=Hom(E,F)\oplus A$ with anchor $(f,x)\mapsto\rho(x)$, hom function $(Id,\Theta)$ and bracket
$$[(f,X),(g,Y)]=([f,g]_{\partial}+\nabla^{\alpha}_Xf-\nabla^{\alpha}_Yg+K(X,Y),[X,Y]_A).$$
\end{enumerate}
\end{thm}
\begin{proof}
It is enough to find that the Jacobi identity for the bracket of $\tilde{A}$.
\begin{align}
&[(f,\Theta(X)),[(g,Y),(h,Z)]]+[(h,\Theta(Z)),[(f,X),(g,Y)]]\nonumber\\
&+[(g,\Theta(Y)),[(h,X),(h,Z)]]\nonumber\\
&=[(f,\Theta(X)),([g,h]_{\partial}+\nabla^{\alpha}_Yh-\nabla^{\alpha}_Zg+K(Y,Z),[Y,Z]_A)]\nonumber\\
&+[(h,\Theta(Z)),([f,g]_{\partial}+\nabla^{\alpha}_Xg-\nabla^{\alpha}_Yf+K(X,Y),[X,Y]_A)]\nonumber\\
&+[(g,\Theta(Y)),([h,f]_{\partial}+\nabla^{\alpha}_Zf-\nabla^{\alpha}_Xh+K(Z,X),[Z,X]_A)]\nonumber
\end{align}
\begin{align}
&=\Bigg([f,[g,h]_{\partial}]_{\partial}+[f,\nabla^{\alpha}_Yh]_{\partial}-[f,\nabla^{\alpha}_Zg]_{\partial}+[f,K(Y,Z)]_{\partial}\nonumber\\
&+\nabla^{\alpha}_{\Theta(X)}[g,h]+\nabla^{\alpha}_{\Theta(X)}\nabla^{\alpha}_Yh-\nabla^{\alpha}_{\Theta(X)}\nabla^{\alpha}_Zg
+\nabla^{\alpha}_{\Theta(X)}K(Y,Z)\nonumber\\
&-\nabla^{\alpha}_{[Y,Z]}f+K(\Theta(X),[Y,Z])\mbox{{\LARGE ,}}[\Theta(X),[Y,Z]_A]_A\Bigg)\nonumber\\
&+\Bigg([h,[f,g]_{\partial}]_{\partial}+[h,\nabla^{\alpha}_Xg]_{\partial}-[h,\nabla^{\alpha}_Yf]_{\partial}+[h,K(X,Y)]_{\partial}\nonumber\\
&+\nabla^{\alpha}_{\Theta(Z)}[f,g]+\nabla^{\alpha}_{\Theta(Z)}\nabla^{\alpha}_Xg-\nabla^{\alpha}_{\Theta(Z)}\nabla^{\alpha}_Yf
+\nabla^{\alpha}_{\Theta(Z)}K(X,Y)\nonumber\\
&-\nabla^{\alpha}_{[X,Y]}f+K(\Theta(Z),[X,Y])\mbox{{\LARGE ,}}[\Theta(Z),[X,Y]_A]_A\Bigg)\nonumber\\
&=\Bigg([g,[h,f]_{\partial}]_{\partial}+[g,\nabla^{\alpha}_Zf]_{\partial}-[g,\nabla^{\alpha}_Xh]_{\partial}+[g,K(Z,X)]_{\partial}\nonumber\\
&+\nabla^{\alpha}_{\Theta(Y)}[h,f]+\nabla^{\alpha}_{\Theta(Y)}\nabla^{\alpha}_Zf-\nabla^{\alpha}_{\Theta(Y)}\nabla^{\alpha}_Xh
+\nabla^{\alpha}_{\Theta(Y)}K(Z,X)\nonumber\\
&-\nabla^{\alpha}_{[Z,X]}g+K(\Theta(Y),[Z,X])\mbox{{\LARGE ,}}[\Theta(Y),[Z,X]_A]_A\Bigg)\nonumber\\
&=\Bigg([f,[g,h]_{\partial}]_{\partial}+[h,[f,g]_{\partial}]_{\partial}+[g,[h,f]_{\partial}]_{\partial}\label{1}\\
&+\nabla^{\alpha}_{\Theta(X)}[g,h]_{\partial}+[h,\nabla^{\alpha}_Xg]_{\partial}-[g,\nabla^{\alpha}_Xh]_{\partial}\label{2}\\
&+\nabla^{\alpha}_{\Theta(Z)}[f,g]_{\partial}+[g,\nabla^{\alpha}_Zf]_{\partial}-[f,\nabla^{\alpha}_Zg]_{\partial}\label{3}\\
&+\nabla^{\alpha}_{\Theta(Y)}[h,f]_{\partial}+[f,\nabla^{\alpha}_Yh]_{\partial}-[h,\nabla^{\alpha}_Yf]_{\partial}\label{4}\\
&+[f,K(Y,Z)]_{\partial}-\nabla^{\alpha}_{[Y,Z]_A}f+\nabla^{\alpha}_{\Theta(Y)}\nabla^{\alpha}_Zf-\nabla^{\alpha}_{\Theta(Z)}\nabla^{\alpha}_Yf\label{5}\\
&+[h,K(X,Y)]_{\partial}-\nabla^{\alpha}_{[X,Y]_A}h+\nabla^{\alpha}_{\Theta(X)}\nabla^{\alpha}_Yh-\nabla^{\alpha}_{\Theta(Y)}\nabla^{\alpha}_Xh\label{6}\\
&+[g,K(Z,X)]_{\partial}-\nabla^{\alpha}_{[Z,X]_A}g+\nabla^{\alpha}_{\Theta(Z)}\nabla^{\alpha}_Xg-\nabla^{\alpha}_{\Theta(X)}\nabla^{\alpha}_Zg\label{7}\\
&+K(\Theta(X),[Y,Z]_A)+K(\Theta(Z),[X,Y]_A)+K(\Theta(Y),[Z,X]_A)\label{8}\\
&+\nabla^{\alpha}_{\Theta(X)}K(Y,Z)+\nabla^{\alpha}_{\Theta(Z)}K(X,Y)+\nabla^{\alpha}_{\Theta(Y)}K(Z,X)\label{9}\\
&\mbox{{\LARGE ,}}[\Theta(X),[Y,Z]_A]_A+[\Theta(Z),[X,Y]_A]_A+[\Theta(Y),[Z,X]_A]_A\Bigg)\label{10}
\end{align}
(\ref{1}) and (\ref{10}) are vanished by Jacobi identity and hom-Jacobi identity, (\ref{2}), (\ref{3}) and (\ref{4}) are zero by compatibility
of $\partial$ and $\nabla^{\alpha}$.
(\ref{5}), (\ref{6}) and (\ref{7}) are vanished by (\ref{r}), at long last $\mbox{Eq. \ref{8}}+\mbox{Eq. \ref{9}}=0$ since $K$ is close.
\end{proof}
\section*{Authors contributions}
This work was carried out by the two authors, in collaboration. All authors participated in the research of this work and performed equally. S. Merati wrote the paper and M. R. Farhangdoost edited it.
All authors read and approved the final manuscript.
\subsection*{Acknowledgements}
The authors would like to thank Shiraz University, for financial support, which leads to the formation of this manuscript.
This research is supported by grant no. 92grd1m82582 in Shiraz University, Shiraz, Iran.
\subsection*{Compliance with ethical guidelines}
\textbf{Competing interests} The authors declare that they have no competing interests.


\begin{thebibliography}{99}

\bibitem{2} C. A. Abad, M. Crainic.
\emph{Representations up to homotopy of Lie algebroids},
J. reine angew. Math. 663, 2012, 91-126.
\bibitem{9} M. Crainic,
R. L. Fernandes.
\emph{Lectures on integrability of Lie
brackets},
Geom. Topol. Monogr. 17, 2011, 1–107.

\bibitem{11} C. L. Gengoux, J. Teles.
\emph{Hom-Lie algebroids},
J. Geom. Phys. 68, 2013, 69-75.
\bibitem{10}J. Hartwig, D. Larsson, S. Silvestrov.
\emph{Deformations of Lie algebras using $\sigma$-derivations},
J. Algebra 295, 2006, 314-361.

\bibitem{12} A. Kitouni, A. Makhlouf, S. Silvestrov.
\emph{On (n+1)-Hom-Lie algebras induced by n-Hom-Lie algebras},
Georgian. Math. J. 23(1), 2016, 75-95.
\bibitem{1} K. C. H. Mackenzie.
\emph{General theory of Lie groupoids and Lie algebroids},
London Mathematical Society Lecture Note Series. Cambridge University Press, 213, Cambridge, 2005.

\bibitem{3} K. C. H. Mackenzie.
\emph{Double Lie algebroids and second-order geometry I},
Adv. Math. 94(2), 1992, 180-239.

\bibitem{6} K. C. H. Mackenzie.
\emph{Double Lie algebroids and second-order geometry II},
Adv. Math. 154, 2000, 46-75.
\bibitem{8} A. Makhlouf, S. Silvestrov.
\emph{Hom-algebra structures},
J. Gen. Lie Theory Appl. 2(2), 2008, 51-64.

\bibitem{4} R. A. Mehta.
\emph{Lie algebroid modules and representations up to
homotopy},
Indag. Math. 25, 2014, 1122-1134.

\bibitem{7} Y. Sheng.
\emph{Representations of hom-Lie algebras},
Algebr. Represent. Theor. 15, 2012, 1081-1098.
\bibitem{5} A. Yu. Va$\check{i}$ntrob.
\emph{Lie algebroids and homological vector fields},
Uspekhi Mat. Nauk. 52(2(314)), 1997, 161-162.

\end{thebibliography}
\end{document}